\documentclass[12pt,oneside]{amsart}
\usepackage{amssymb,verbatim,enumerate,amsmath,ifthen,cite}
\usepackage[mathscr]{eucal}
\usepackage[utf8]{inputenc}
\usepackage[T1]{fontenc}
\usepackage[marginparwidth=2.4cm]{geometry}
\newtheorem{theorem}{Theorem}[section]
\newtheorem*{theorem*}{Theorem}
\def\Thm#1#2{\ifthenelse{\equal{#1}{*}}{\begin{theorem*}#2\end{theorem*}}
{\begin{theorem}\label{T#1}#2\end{theorem}}}
\newtheorem{Atheorem}{Theorem}

\def\thm#1{Theorem~\ref{T#1}}

\newtheorem{proposition}[theorem]{Proposition}
\newtheorem*{proposition*}{Proposition}
\def\Prp#1#2{\ifthenelse{\equal{#1}{*}}{\begin{proposition*}#2\end{proposition*}}
{\begin{proposition}\label{P#1}#2\end{proposition}}}
\def\prp#1{Proposition~\ref{P#1}}

\newtheorem{corollary}[theorem]{Corollary}
\newtheorem*{corollary*}{Corollary}
\def\Cor#1#2{\ifthenelse{\equal{#1}{*}}{\begin{corollary*}#2\end{corollary*}}
{\begin{corollary}\label{C#1}#2\end{corollary}}}

\newtheorem{lemma}[theorem]{Lemma}
\newtheorem*{lemma*}{Lemma}
\def\Lem#1#2{\ifthenelse{\equal{#1}{*}}{\begin{lemma*}#2\end{lemma*}}
{\begin{lemma}\label{L#1}#2\end{lemma}}}
\def\lem#1{Lemma~\ref{L#1}}
\newtheorem{Alemma}{Lemma}

\theoremstyle{definition}
\newtheorem{remark}[theorem]{Remark}
\newtheorem*{remark*}{Remark}
\def\Rem#1#2{\ifthenelse{\equal{#1}{*}}{\begin{remark*}\rm #2\end{remark*}}
{\begin{remark}\label{R#1}\rm #2\end{remark}}}
\def\rem#1{Remark~\ref{R#1}}

\newtheorem{example}[theorem]{Example}
\newtheorem*{example*}{Example}
\def\Exa#1#2{\ifthenelse{\equal{#1}{*}}{\begin{example*}\rm #2\end{example*}}
{\begin{example}\label{Ex#1}\rm #2\end{example}}}

\numberwithin{equation}{section}
\def\eq#1{{\rm(\ref{#1})}}
\def\Eq#1#2{\ifthenelse{\equal{#1}{*}}
  {\begin{equation*}\begin{aligned}[]#2\end{aligned}\end{equation*}}
  {\begin{equation}\begin{aligned}[]\label{#1}#2\end{aligned}\end{equation}}}

\def\A{\mathscr{A}}

\def\calF{\mathcal{F}}

\def\Cts{\mathcal{C}^{2\#}}
\def\Cs{\mathcal{C}^{1\#}}
\def\CM{\mathcal{CM}}

\newcommand{\UQA}[1]{\mathscr{U}_{#1}}
\newcommand{\LQA}[1]{\mathscr{L}_{#1}}

\newcommand\R{\mathbb{R}}
\newcommand\N{\mathbb{N}}

\newcommand{\QA}[1]{\A_{#1}}

\DeclareMathOperator{\interior}{int}

\newcommand{\abs}[1]{\left| #1 \right|}

\begin{document}
\title[Lattice-like property of quasi-arithmetic means: revisited]{Lattice-like property of quasi-arithmetic means: revisited}

\author[Tibor Kiss]{Tibor Kiss}
\address{%
Institute of Mathematics, 
University of Debrecen,
Pf.\ 400, 4002 Debrecen, Hungary}
\email{kiss.tibor@science.unideb.hu}
\author{Pawe{\l} Pasteczka}
\address{Institute of Mathematics, University of Rzesz\'ow, Pigonia 1, 35-310 Rzesz\'ow, Poland}
\email{ppasteczka@ur.edu.pl}
\subjclass{Primary TBD; Secondary TBD}

\keywords{quasi-arithmetic means; differentiability; comparability; conditionally complete lattice}

\subjclass{%
26E60, 
06B35, 
06B23, 
03G10
}

\thanks{The research of the first author was supported by the Tempus Public Foundation and HUN-REN Hungarian Research Network.}
\begin{abstract}
We show that every family of quasi-arithmetic means generated by (a subset of) $\mathcal{C}^1$ functions with nonvanishing derivative which is bounded (from below or from above) by a quasi-arithmetic mean, possesses the best (lower or upper)  bound which is a quasi-arithmetic mean generated by a function belonging to the same family.
\end{abstract}

\maketitle

\section{Introduction}
Quasi-arithmetic means were introduced as a generalization of power means in the 1920s and 1930s in a group of almost simultaneous papers \cite{Kno28,Def31,Kol30,Nag30}.  For the recent results concerning axiomatization of this family we refer the reader to the papers by Burai-Kiss-Szokol \cite{BurKisSzo21,BurKisSzo23}.

Given an interval $I$ and a continuous, strictly monotone function $f \colon I \to \R$ (from now on, $I$ will always denote a nontrivial interval and $\CM(I)$ the family of continuous, strictly monotone functions on $I$), the \emph{quasi-arithmetic mean} $\QA{f} \colon \bigcup_{n=1}^\infty I^n \to I$ is defined by
\Eq{*}{
\QA{f}(v):=f^{-1}\left( \frac{f(v_1)+f(v_2)+\cdots+f(v_n)}{n} \right) \quad (n \in \N,\: v \in I^n)\:.
}
The function $f$ is referred to as the \emph{generator} of the quasi-arithmetic mean $\QA{f}$.

It is well known that, for $I=\R_+$, $\pi_p(x):=x^p$ when $p\ne 0$ and $\pi_0(x):=\ln x$, the mean $\QA{\pi_p}$ coincides with the $p$-th power mean (this was already observed by Knopp \cite{Kno28}).

Numerous results concerning quasi-arithmetic means have been obtained in various contexts; see, among others, the works of  
Baj\'ak-P\'ales \cite{BajPal09a}, 
Burai \cite{Bur13a},
Cargo-Shisha \cite{CarShi64,CarShi69},
Dar\'oczy \cite{Dar04a,Dar07}, 
Dar\'oczy-Hajdu \cite{DarHaj05}, 
Dar\'oczy-Maksa \cite{DarMak13}, 
Dar\'oczy-P\'ales \cite{DarPal01a,DarPal02b,DarPal03a,DarPal03d,DarPal08,DarPal13}, 
G{\l}azowska-W.Jarczyk-Matkowski \cite{GlaJarMat02}, 
Ito-Nara \cite{ItoNar86}, 
J.~Jarczyk \cite{Jar07a,Jar07b,Jar09b}, 
W.~Jarczyk-Matkowski \cite{JarMat00}, 
J.~Jarczyk-Matkowski \cite{JarMat06}, 
Kahlig-Matkowski \cite{KahMat97b},
Maksa-P\'ales \cite{MakPal09a,MakPal10a}, 
Marichal \cite{Mar00}, 
Matkowski \cite{Mat03d,Mat99a,Mat10b}, 
Matkowski-P\'ales \cite{MatPal15}, 
P\'ales \cite{Pal11,Pal86b,Pal87d},
P\'ales-Pasteczka \cite{PalPas26},
and Pasteczka \cite{Pas19b,Pas20b}.

For instance, one can introduce a preorder on $\CM(I)$ by
\Eq{*}{
f \vartriangleleft g \iff \QA{f}(v) \le \QA{g}(v) \text{ for all } v \in \bigcup_{n \in \N}I^n.
}
It is a classical fact (see, e.g., \cite{HarLitPol34}) that $\QA{f}=\QA{g}$ holds if and only if the generators are affinely related, i.e., there exist $\alpha \in \R \setminus\{0\}$ and $\beta \in \R$ such that $g=\alpha f + \beta$. Consequently, it is natural to introduce an equivalence relation $\sim$ on $\CM(I)$ by
\Eq{*}{
f \sim g \iff \QA{f} = \QA{g}\:.
}

Moreover, one easily verifies that this preorder induces a partial order on the quotient set $\CM(I) /_\sim$. This order exhibits numerous interesting properties (see, for example, the results of Cargo-Shisha \cite{CarShi64,CarShi69}).

As we intend to study certain lattice properties of quasi-arithmetic means, we first need to introduce the concepts of the supremum and infimum of an arbitrary subfamily of this class (for a thorough introduction to lattice theory, see e.g. \cite{DavPri02,Fel13,AndFei88}). 
In this paper we adapt several conventions from \cite{Pas20b}. 

To begin with, for a single element $f \in \CM(I)$, we consider the set of all generators of quasi-arithmetic means that dominate the one generated by $f$, namely the order interval
\[
U_f := \{ s \in \CM(I) : \QA{f}\le\QA{s} \}.
\]
Next, for any subfamily $\calF \subset \CM(I)$ we define a function $\UQA{\calF} \colon \bigcup_{n=1}^\infty I^n \to I$ by
\Eq{*}{
\UQA{\calF}(v) := 
\begin{cases}
\inf \big\{\QA{s}(v) : s \in \bigcap\limits_{f \in \calF} U_f\big\} & \text{ if }\bigcap\limits_{f \in \calF} U_f \neq \emptyset,\\[5mm]
\max(v) & \text{ otherwise.}
\end{cases}
}
Observe that $\UQA{\calF}$ is indeed a mean (it is defined on $\bigcup_{n=1}^\infty I^n$ and satisfies $\min \le \UQA{\calF} \le \max$), whereas $U_f$ is a set of functions. In an analogous way we introduce
\[
L_f := \{ s \in \CM(I) : \QA{f} \ge \QA{s}\},
\]
and, for $\calF \subset \CM(I)$, we define a mean $\LQA{\calF} \colon \bigcup_{n=1}^\infty I^n \to I$ by
\Eq{*}{
\LQA{\calF}(v) := 
\begin{cases}
\sup \big\{\QA{s}(v) : s \in \bigcap\limits_{f \in \calF} L_f\big\} & \text{ if }\bigcap\limits_{f \in \calF} L_f \neq \emptyset,\\[5mm]
\min(v) & \text{ otherwise.}
\end{cases}
}

Clearly, for each subfamily $\calF$, the mappings $\LQA{\calF}$ and $\UQA{\calF}$ are monotone and symmetric (on every $I^n$). Moreover, whenever they turn out to be quasi-arithmetic means, their generators are, respectively, the infimum and supremum of the family $\calF$ with respect to the partial order $\prec$. This naturally leads to the problem of finding sufficient conditions that guarantee that $\LQA{\calF}$ and $\UQA{\calF}$ are quasi-arithmetic means, and, once this is ensured, determining their generators.

Hereafter let $\mathcal{C}^{k\#}(I)$ stand for the family of all $\mathcal{C}^{k}$ functions on $I$ with nowhere vanishing first derivative. Clearly all functions belonging to such a family are strictly monotone and therefore they generate quasi-arithmetic means. Several results concerning the family $\Cs(I)$ are enclosed in the paper \cite{Pas19b}.

In our main result we show that every family of quasiarithmeic means generated by (a subset of) $\Cs(I)$ functions which is bounded by a quasiarithmeic mean possess the best bound ($\LQA{\calF}$ or $\UQA{\calF}$) which is a quasiarithmeic mean generated by a $\Cs(I)$ function.

\section{Preliminary results}
The following is the main result of \cite{Pas20b} which we are going to generalize.

\begin{proposition}[Pasteczka \cite{Pas20b}, Theorem~1]
 Let $\calF \subset \Cts(I)$. 
 \begin{enumerate}
 \item If the function $G \colon I \ni x \mapsto \sup_{f \in \calF} \frac{f''(x)}{f'(x)}$ is continuous, then $\UQA \calF=\QA{g}$, where $g \in \Cts(I)$ and $g''/g'=G$.
  \item If the function $H \colon I \ni x \mapsto \inf_{f \in \calF} \frac{f''(x)}{f'(x)}$ is continuous, then $\LQA \calF=\QA{h}$, where $h \in \Cts(I)$ and $h''/h'=H$.
  \end{enumerate}
 \end{proposition}

Let us also recall several technical lemmas.
\begin{lemma}[P\'ales \cite{Pal91}, Corollary~1]
 \label{LPal91}
Let $f$ and $f_n$, $n \in \N$, be continuous, strictly monotone functions defined on $I$. Then $\lim\limits_{n \to \infty} \QA{f_n}=\QA{f}$ pointwise if and only if
\Eq{*}{
\lim_{n \to \infty} \frac{f_n(x)-f_n(z)}{f_n(y)-f_n(z)}=\frac{f(x)-f(z)}{f(y)-f(z)} \qquad \text{for all }x,\,y,\,z \in I \text{ with } y \ne z.
}
\end{lemma}

Now let us recall some recent results concerning smoothness properties implied by comparability enclosed in \cite{Pas19b}.

\begin{lemma}[Pasteczka \cite{Pas19b}]
 \label{Lhnonvanish}
 Let $f,\,g \in \CM(I)$ such that $\QA{f} \le \QA{g}$. 
 Then 
 \begin{itemize}
  \item one-sided derivatives $f'_-$ and $g'_-$  (resp. $f'_+$ and $g'_+$) exist at the same points;
  \item $f'_-$ and $g'_-$ (resp. $f'_+$ and $g'_+$) vanish at the same points.
 \end{itemize}
\end{lemma}

\begin{lemma}[Pasteczka \cite{Pas19b}]
\label{Lsandwich}
 Let $f,\,g,\,h \in \CM(I)$ such that $\QA{f} \le \QA{g}\le \QA{h}$. 
\begin{enumerate}[(a)]
\item If $f$ and $h$ are differentiable at a certain point $x_0 \in I$ then so is $g$.
\item If $f, h \in \Cs(I)$ then so is $g$.
\end{enumerate}
\end{lemma}

\Rem{C1QA}{
Let $f,g \in \Cs(I)$ be strictly increasing. Then $\QA{f}\le\QA{g}$ if and only if $\frac{g'}{f'}$ is increasing, or equivalently, $\log g'-\log f'$ is  increasing.
}

Finally we show a result concerning convexity/concavity of a real-valued function.

\Lem{concave}{
Let $I \subset \R$ be an open interval and $f \colon I \to \R$ be an arbitrary function. Assume that there exists a finite set $Z$ such that $f$ is differentiable at every point of $Z$ and $f$ is convex (resp. concave) on every connected component of $I \setminus Z$. Then $f$ is convex (resp. concave).
}
\begin{proof}
Since convexity is a localizable property, it is sufficient to show that $f$ is convex in a neighborhood of every point of $x_0\in I$. If $x_0 \in I \setminus Z$ then $f$ is convex on a connected component of $I \setminus Z$ containing $x_0$, and we are done. Otherwise there exists $\varepsilon>0$ such that $(x_0-\varepsilon,x_0+\varepsilon)\cap Z=\{x_0\}$. Then $f$ is convex on $(x_0-\varepsilon,x_0)$ and $(x_0,x_0+\varepsilon)$ and differentiable at $x_0$. In particular, one-sided derivatives of $f$ are nondecreasing on $(x_0-\varepsilon,x_0+\varepsilon)$. This proves that $f$, being a primitive of a nondecreasing function, is convex on this interval.
\end{proof}

\section{Main results}

Let $f,g \colon I \to \R$ be continuous functions. Define the order $f \prec g$ if and only if $f-g$ is nonincreasing. We begin with establishing its most important properties. Recall that a lattice is said to be \emph{conditionally complete} if any nonempty subset bounded above has the least upper bound and any nonempty subset bounded below has the greatest lower bound.

\Prp{CC}{
The ordered set $(\mathcal{C}(I),\prec)$ is a conditionally complete lattice.
}
\begin{proof}
Let $\calF=\{f_\gamma \colon I \to \R \mid \gamma \in \Gamma\}$ be a nonempty family of continuous functions. Assume that there exists a continuous function $g \colon I \to \R$ such that $f_\gamma \prec g$ for all $\gamma \in \Gamma$. Define $\delta \colon \{(x,y)\in I^2\colon x \le y\} \to \R$ by  
\Eq{*}{
\delta(x,y):=\inf_{\gamma \in \Gamma} f_\gamma(x)-f_\gamma(y), \qquad x,y\in I,\ x\le y.
}
Observe that $k \colon I \to \R$ is an upper bound of $\calF$ if and only if
\Eq{ub}{
k(x)-k(y) \le \delta(x,y)\quad \text{ for all }x,y\in I,\ x\le y.
}

Next we define $\Delta \colon \{(x,y)\in I^2\colon x \le y\} \to \R$ by  
\Eq{*}{
\Delta(x,y):=\inf \Big\{\sum_{i=1}^n \delta(t_{i-1},t_i) \colon n \in \N \text{ and } x=t_0<t_1<\dots<t_n=y\Big\}.
}
First, we show that $\Delta$ is indeed a real-valued function. Using that $g$ is an upper bound of $\calF$, for arbitrary partition $(t_i)_{i=0}^n$ of $[x,y]$, we can write $g(t_{i-1})-g(t_i)\leq \delta(t_{i-1},t_i)$ for all $i=1,\dots,n$. Summing up these inequalities side-by-side, we obtain that
\Eq{*}{g(x)-g(y)\leq \sum\limits_{i=1}^n\delta(t_{i-1},t_i).}
Since the above inequality is true for any partition of $[x,y]$, it follows that $\Delta(x,y)\in\R$ and 
\Eq{gg}{
g(x)-g(y)\le \Delta(x,y)\text{ for all }x,y\in I\text{ with }x\leq y.
}
Let us note that adding an extra point cannot increase the sum of $\delta$-s, namely we have
\Eq{*}{
\Delta(x,z)=\inf \Big\{ \sum_{j=-m+1}^n \delta(t_{j-1},t_j) \colon m,n \in \N\text{ and }x=t_{-m}<\dots<t_0=y<t_1<\dots<t_n=z\Big\}.
}

Consequently, $\Delta$ is well-defined. We also note that this computation yields that whenever $k:I\to\R$ is an upper bound of $\calF$, it follows that $k(x)-k(y)\leq\Delta(x,y)$ holds true for $x,y\in I$ with $x\leq y$.

Observe that for all $x,y,z\in I$ with $x\leq y\leq z$ we have $\delta(x,y)+\delta(y,z)\le \delta(x,z)$. In the next part of the proof, we show that, in the case of $\Delta$, this condition already upgrades to equality. More precisely, we show that 
\Eq{Deltaxyz}{
\Delta (x,y)+\Delta(y,z)= \Delta(x,z),\qquad \text{ for all }x,y,z \in I\text{ with }x\leq y\leq z.
}
Let $P_1=(t_i)_{i=-m}^0$ and $P_2=(t_i)_{i=0}^n$ be partitions of $[x,y]$ and $[y,z]$, respectively, with $t_{-m}=x$, $t_0=y$ and $t_n=z$. Then
\Eq{*}{
\sum_{j=-m+1}^n \delta(t_{j-1},t_j)=\sum_{j=-m+1}^0 \delta(t_{j-1},t_j)+\sum_{j=1}^n \delta(t_{j-1},t_j)\ge \Delta (x,y)+\Delta(y,z).
}
Taking the infimum on the left-hand side, we obtain $\Delta(x,z) \ge \Delta (x,y)+\Delta(y,z)$.

To establish the inequality in the reverse direction, fix $\varepsilon>0$ arbitrarily. There exists a partition $(t_i)_{i=-m}^n$ of $[x,z]$ with $t_{-m}=x$, $t_0= y$ and $t_n=z$,  such that 
\Eq{*}{
\sum_{j=-m+1}^0 \delta(t_{j-1},t_j) \le \Delta (x,y)+\varepsilon \qquad 
\sum_{j=1}^n \delta(t_{j-1},t_j)\le \Delta(y,z)+\varepsilon.
}
Summing up the above inequalities side-by-side we get
\Eq{*}{
\sum_{j=-m+1}^n \delta(t_{j-1},t_j) \le \Delta (x,y)+\Delta(y,z)+2\varepsilon.
}
On the other hand 
\Eq{*}{
\sum_{j=-m+1}^n \delta(t_{j-1},t_j) \ge \Delta(x,z).
}
Therefore $\Delta (x,y)+\Delta(y,z)+2\varepsilon \ge \Delta(x,z)$. Taking the limit $\varepsilon \to 0^+$ we get $\Delta (x,y)+\Delta(y,z)\ge \Delta(x,z)$. Thus \eq{Deltaxyz} holds.

We are now in a position to define the function $h$. Let $x_0 \in I$ be an arbitrarily fixed point in the interior of $I$. We define $h \colon I \to \R$ by
\Eq{*}{
h(x):=
\begin{cases}
\Delta(x,x_0)&\text{ if }x \le x_0, \\
-\Delta(x_0,x)&\text{ if }x_0<x.
\end{cases}
}
We claim that $h$ is continuous with $h=\sup_{\prec}\{f_\gamma \colon I \to \R \mid \gamma \in \Gamma\}$. First, let us focus on the supremum property. We show that $h$ is an upper bound of $\calF$, and then that it is no greater than any other upper bound.

We show that $h$ fulfils \eq{ub}. Using \eq{Deltaxyz} we verify this inequality in three cases. 
If $x\le y \le x_0$ then we have
\Eq{*}{
h(x)=\Delta(x,x_0)= \Delta(x,y)+\Delta(y,x_0) \le \delta(x,y)+\Delta(y,x_0)=\delta(x,y)+h(y).
}
In the second case, if $x\le x_0< y$, we have that
\Eq{*}{
h(x)=\Delta(x,x_0)=\Delta(x,y)-\Delta(x_0,y)\le \delta(x,y)-\Delta(x_0,y)=\delta(x,y)+h(y).
}
Finally, if $x_0<x\le y$ then
\Eq{*}{
h(x)=-\Delta(x_0,x)=\Delta(x,y)-\Delta(x_0,y)\le \delta(x,y)-\Delta(x_0,y)=\delta(x,y)+h(y).
}
Therefore $h$ fulfils \eq{ub}.

Based on a remark made earlier in the proof, if $k \colon I \to \R$ is an upper bound, then $k(x)-k(y)\le \Delta(x,y)$. Now it is enough to observe that
\Eq{*}{
\Delta(x,y)=
\begin{cases}
\Delta(x,x_0)-\Delta(y,x_0)&\text{if }x\leq y\le x_0,\\[1mm]
\Delta(x,x_0)+\Delta(x_0,y)&\text{if }x\le x_0 <y\\[1mm]
\Delta(x_0,y)-\Delta(x_0,x)&\text{if }x_0<x\leq y,
\end{cases}}
where on the right hand side $h(x)-h(y)$ appears in each instance. In this way, we proved that
\Eq{*}{
k(x)-k(y)\le h(x)-h(y) \quad \text{ whenever }x<y,
}
consequently $h \prec k$, and therefore $h$ is the $\prec$-supremum of $\calF$.

In the last step, we prove that $h$ is also continuous. By \eq{gg}, for all $u,v \in I$ with $u<v$, and $\gamma_0 \in \Gamma$ we have 
\Eq{*}{
g(u)-g(v)\le \Delta(u,v)\le \delta(u,v)\le f_{\gamma_0}(u)-f_{\gamma_0}(v).
}
Applying the Squeeze Theorem, we get that
\Eq{*}{
\lim_{u \to v^-} \Delta(u,v)=\lim_{v \to u^+} \Delta(u,v)=0.
}
Let $p\in I$ be any. Then, for $p\le x_0$ and every sequence $x_n \to p^-$, we have
\Eq{*}{
\lim_{n \to \infty}\abs{h(x_n)-h(p)} =\lim_{n \to \infty}\abs{\Delta(x_n,x_0)-\Delta(x_0,p)}=\lim_{n \to \infty}\abs{\Delta(x_n,p)}=0.
}
The remaining cases, that is, the cases ($p<x_0$, $x_n \to p^+$) or ($p> x_0$, $x_n \to p^-$) or ($p\ge x_0$, $x_n \to p^+$) are completely analogous.
\end{proof}

\Thm{mf}{
Let $f \colon I \to \R$ be a continuous, strictly monotone function. Then there exists $m_f \in \Cs(I)$ such that:
\begin{enumerate}[\rm (a)]
\item $\{g \in \Cs(I) \colon \QA{f} \le \QA{g}\}$ is either empty or equals to $\{g \in \Cs(I) \colon \QA{m_f} \le \QA{g}\}$;
\item $\{g \in \Cs(I) \colon \QA{f} \ge \QA{g}\}$ is either empty or equals to $\{g \in \Cs(I) \colon \QA{m_f} \ge \QA{g}\}$.
\end{enumerate}
}
\begin{proof}
If $f \in \Cs(I)$, then this statement is trivially satisfied with $m_f=f$.

Assume that $f \notin \Cs(I)$. Denote 
\Eq{*}{
A:= \{g \in \Cs(I) \colon \QA{f} \le \QA{g}\},\qquad B:=\{g \in \Cs(I) \colon \QA{f} \ge \QA{g}\}.
}
If $f$ is not one-sided differentiable at every point, then, due to \lem{hnonvanish}, we have $A=B=\emptyset$. Therefore, we can assume that $f'_-$ and $f'_+$ exist. If there exists $x_0 \in I$ such that $f'_-(x_0)=0$ or $f'_+(x_0)=0$ then, again due to \lem{hnonvanish}, we have $A=B=\emptyset$.

For the remaining part of the proof, we assume that for all $x \in I$ we have $f'_-(x) >0$ and $f'_+(x) >0$. 
If $A \ne \emptyset$ and $B \ne \emptyset$ then there exists $g_1,g_2 \in \Cs(I)$ such that $\QA{g_1}\le\QA{f} \le \QA{g_2}$. Then, due to \lem{sandwich}, we have $f \in \Cs(I)$.

Hence we can assume that precisely one of sets $A$, $B$ is nonempty, say $A \ne \emptyset$ and $B=\emptyset$. We show that there exists $m_f  \in \Cs(I)$ such that $A=\{g \in \Cs(I) \colon \QA{m_f} \le \QA{g}\}$.

Take an arbitrary function $g \in A$. Then we have $\A_f \le \A_{g}$. We can assume that both $f$ and $g$ are strictly increasing.   

Since $\A_f \le \A_{g}$ we get that $f \circ g^{-1}$ is concave. Therefore, 
\Eq{*}{
(f \circ g^{-1})'_-(y) \ge (f \circ g^{-1})'_+(y) \text{ for all }y \in g(I).
}
Thus, since $g^{-1}$ is differentiable, for all $y \in g(I)$ we get 
\Eq{*}{
f'_-(g^{-1}(y))(g^{-1})'(y)=(f \circ g^{-1})'_-(y) \ge (f \circ g^{-1})'_+(y)=f'_+(g^{-1}(y))(g^{-1})'(y).
}
Therefore $f'_-(x) \ge f'_+(x)$ for all $x \in I$. 

Fix $x_0 \in I$ arbitrarily such that $f$ is differentiable at $x_0$. There exist countable sets $Z_+=(z_n)_{n=1}^\infty$ of elements in $(x_0,\infty) \cap I$ and $Z_-=(z_n^*)_{n=1}^\infty$ of elements in $(-\infty,x_0) \cap I$ such that $f$ is differentiable on $I \setminus (Z_+ \cup Z_-)$.

Set $f_1:=f$ and define, for all $n \ge 1$,
\Eq{*}{
f_{n+1}(x):=\begin{cases}
            \dfrac{(f_n)'_+(z_n^*)}{(f_n)'_-(z_n^*)}(f_{n}(x)-f_{n}(z_n^*))+f_{n}(z_n^*) & \text{ for }x <z_n^*,\\[4mm]
            f_n(x) & \text{ for }x \in [z_n^*, z_n],\\
                        \dfrac{(f_n)'_-(z_n)}{(f_n)'_+(z_n)}(f_{n}(x)-f_{n}(z_n))+f_{n}(z_n) & \text{ for }x >z_n.
           \end{cases}
}
Then $f_{n_0}$ is differentiable at $I \setminus ((z_n)_{n \ge n_0} \cup (z_n^*)_{n \ge n_0} )$, for all $n_0 \in \N$.

Suppose that $\A_{f_n} \le \A_{g}$. Then $f_n \circ g^{-1}$ is concave. Thus $f_{n+1}\circ g^{-1}$ is concave (separately) on the intervals $g(I) \cap (-\infty,g(z_{n}^*))$, $(g(z_{n}^*),g(z_{n}))$, and   $g(I) \cap (g(z_{n}),+\infty)$. Furthermore
\Eq{*}{
(f_{n+1}\circ g^{-1})'_+(g(z_n))&=(f_{n+1})'_+(z_n) \cdot  (g^{-1})'(g(z_n))=(f_{n})'_+(z_n) \cdot  (g^{-1})'(g(z_n)),\\
(f_{n+1}\circ g^{-1})'_-(g(z_n))&=(f_{n+1})'_-(z_n) \cdot  (g^{-1})'(g(z_n))=\dfrac{(f_n)'_+(z_n)}{(f_n)'_-(z_n)}(f_{n})'_-(z_n) \cdot  (g^{-1})'(g(z_n)).
}
Similarly $(f_{n+1}\circ g^{-1})'_-(g(z_n^*))=(f_{n+1}\circ g^{-1})'_+(g(z_n^*))$. Thus $f_{n+1}\circ g^{-1}$ is differentiable at $g(z_n)$ and $g(z_n^*)$.
Hence, in view of \lem{concave}, the function $f_{n+1}\circ g^{-1}$ is concave. Therefore $\QA{f_{n+1}} \le \QA{g}$.

Since $\QA{f_1}=\QA{f} \le \QA{g}$ we obtain  
\Eq{297}
{
\QA{f_n} \le \QA{g}\text{ for all }n \in \N.
}

Furthermore, by $f'_-\ge f'_+$, the mapping $n \mapsto f_n(x)$ is pointwise increasing for all $x \in I \cap (-\infty,x_0)$ and pointwise decreasing for all $x \in I \cap (x_0,\infty)$ while the value of $f_n(x_0)$ does not depends on $n$. Therefore $(f_n)_{n=1}^\infty$ has a real-valued pointwise limit, which we denote by $m_f \colon I \to \R$. Clearly $m_f$ is nondecreasing.

We will show that $m_f$ is continuous and invertible. 

Since, for every $n \in \N$, $f_n \circ g^{-1}$ is a concave function on $g(I)$, its pointwise limit $m_f \circ g^{-1}$ is also a concave function on $g(I)$. In particular $m_f \circ g^{-1}$ is continuous function defined on $g(I)$. Therefore, $m_f$ being a composition of two continuous functions is also continuous. 

Similarly, for every $n \in \N$, the function $g \circ f_n^{-1}$ is convex on its domain. Moreover the sequence $(g \circ f_n^{-1})_{n \in \N}$ is pointwise monotone (as a function of $n$). In particular, it has a convex (and therefore continuous) limit $u$ defined on a topological limit of $(f_n^{-1}(I))$, i.e., on $m_f^{-1}(I)$. Then $m_f^{-1}=g^{-1} \circ u$, consequently $m_f$ is invertible. Therefore, the function $m_f$ (being invertible and continuous) is strictly monotone. 

It follows that $m_f$ generates a quasi-arithmetic mean. Moreover, in view of \lem{Pal91}, we have $\lim_{n \to \infty} \QA{f_n}=\QA{m_f}$. In view of \eq{297} we get $\QA{m_f}\le \QA{g}$. Furthermore, by $f'_-(x) \ge f'_+(x)$, the function 
\Eq{*}{
f_{n+1}\circ f_n^{-1}(y):=\begin{cases}
            \dfrac{(f_n)'_+(z_n^*)}{(f_n)'_-(z_n^*)}(y-f_{n}(z_n^*))+f_{n}(z_n^*) & \text{ for }y <f_n(z_n^*),\\[4mm]
            y & \text{ for }y \in [f_n(z_n^*), f_n(z_n)],\\
                        \dfrac{(f_n)'_-(z_n)}{(f_n)'_+(z_n)}(y-f_{n}(z_n))+f_{n}(z_n) & \text{ for }y >f_n(z_n).
           \end{cases}
}

is convex for all $n \in \N$. It implies, that $\QA{f_{n+1}}\ge \QA{f_n}$, and hence $\QA{m_f} \ge \QA{f_n}$ for all $n \in \N$. In particular, $\QA{m_f} \ge \QA{f}$. 

Since $g$ was an arbitrary element in $A$ we get $A \subset \{g \in \Cs(I) \colon \QA{m_f} \le \QA{g}\}$.
Furthermore, by  $\QA{m_f} \ge \QA{f}$ the opposite inclusion is also valid.

Finally, we utilize \lem{sandwich} to show that $m_f \in \Cs(I) $. Observe that $\QA{f}\le \QA{m_f}\le \QA{g}$, $g \in \Cs(I)$, and $f$ is differentiable on $I \setminus (Z_+ \cup Z_-)$. Therefore, by part (a), we get that $m_f$  is differentiable on $I \setminus (Z_+ \cup Z_-)$. 
Moreover for all $n \in \N$ we have that $\QA{f_{n+1}}\le \QA{m_f}\le \QA{g}$ and hence $f_{n+1}$ is differentiable at $z_{n}$ and $z^*_{n}$. Therefore, using the same statement, $m_f$ is differentiable at $z_{n}$ and $z^*_{n}$ as well. Thus $m_f$ is  differentiable on $I$. Furthermore, by \lem{hnonvanish} we obtain that $m_f'$ is nowhere vanishing. 

Finally we show that $m_f'$ is continuous. Since $m_f \circ g^{-1}$ is concave and differentiable we get $m_f\circ g^{-1} \in \mathcal{C}^1(g(I))$. Thus $m_f'=\frac{(m_f\circ g^{-1})'}{(g^{-1})'} \circ g$ is continuous, that is $m_f \in \Cs(I)$.

The proof in the case $A = \emptyset$ and $B\ne\emptyset$ is similar.
\end{proof}

The following is the main result of this paper.

\Thm{main}{
Let $\calF \subset \Cs(I)$.
\begin{enumerate}[\rm (i)]
\item If there exists $g \in \CM(I)$ such that $\QA{f}\le \QA{g}$ for all $f \in \calF$ then there exists $u \in \Cs(I)$ such that
\begin{enumerate}
\item $\QA{f}\le \QA{u}$ for all $f \in \calF$
\item $\QA{u} \le \QA{k}$ for every $k \in \CM(I)$ such that $\QA{f}\leq \QA{k}$ for all $f \in \calF$.
\end{enumerate}
\item If there exists $g \in \CM(I)$ such that $\QA{f}\ge \QA{g}$ for all $f \in \calF$ then there exists $\ell \in \Cs(I)$ such that
\begin{enumerate}
\item $\QA{f}\ge \QA{\ell}$ for all $f \in \calF$
\item $\QA{\ell} \ge \QA{k}$ for every $k \in \CM(I)$ such that $\QA{f}\ge \QA{k}$ for all $f \in \calF$.
\end{enumerate}
\end{enumerate}
}
\begin{proof}
Let us assume without loss of generality that all elements of $\calF$ as well as $u$ and $k$ are strictly increasing.

We will show only part (i) since (ii) is analogous. Let $\calF \subset \Cs(I)$ and $g \in \CM(I)$ be such that 
 $\QA{f}\le \QA{g}$ for all $f \in \calF$. In view of \thm{mf} there exists a function $m_g \in \Cs(I)$ such that
$\{f \in \Cs(I) \colon \QA{f} \le \QA{g}\}=\{f \in \Cs(I) \colon \QA{f} \le \QA{m_g}\}$. 
Then, by \rem{C1QA}, $\frac{f'}{m_g'}$ is nonincreasing for all $f \in \calF$. Since $m_f'$ and $f'$ are nonvanishing, we get that $\log f'-\log m_g'$ is nonincreasing  for all $f \in \calF$, i.e. $\log f' \prec \log m_g'$. 
 
Then, since $(\mathcal{C}(I),\prec)$ is conditionally complete (vide \prp{CC}), there exists a continuous $s \colon I \to \R$ such that $\log f' \prec s$ for all $f \in \calF$ and $s \prec h$ for every $\prec$-upper bound $h$ of $\calF$. 

Fix $x_0 \in I$ arbitrarily and define $u \colon I \to \R$ by 
\Eq{*}{
u(x)=\int_{x_0}^x e^{s(t)}\:dt.
}
We will show that $u$ satisfies both properties (a) and (b). Note that
$\log u'-\log f'=s-\log f'$ is increasing. Due to the \rem{C1QA} we get $\QA{u} \ge \QA{f}$ which is the proof of the property (a) of (i). 

Now take an arbitrary $k \in \CM(I)$ such that $\QA{f}\le \QA{k}$ for all $f \in \calF$. Due to \thm{mf} there exists $m_k \in \Cs(I)$ such that $\QA{f}\le \QA{m_k}$ for all $f \in \calF$ and $\QA{m_k} \le \QA{k}$.

Therefore $ \log f' \prec \log m_k'$ for all $f \in \calF$.  Thus $\log u'=s\prec \log m_k'$ which, using \rem{C1QA} again, implies that $\QA{u}\le \QA{m_k} \le \QA{k}$. This completes the proof of the property (b) of (i).
\end{proof}

Now we express the same in terms of means $\LQA{}$ and $\UQA{}$.
\Thm{FCs}{
Let $\calF \subset \Cs(I)$. Then
\begin{enumerate}[\rm (i)]
\item $\UQA{\calF}$ equals to $\QA{u}$ for some $u \in\Cs(I)$ or equals to $\max$;
\item $\LQA{\calF}$ equals to $\QA{\ell}$ for some $\ell \in\Cs(I)$ or equals to $\min$.
\end{enumerate}
}
\begin{proof}
Take $\calF \subset \Cs(I)$ such that $\UQA{\calF}$ is not a maximum. Then, by the definition, 
\Eq{*}{
S:=\bigcap\limits_{f \in \calF} U_f=\bigcap\limits_{f \in \calF} \{ s \in \CM(I) \colon \QA{f} \le \QA{s}\} \ne \emptyset,
}
and we have
$\UQA{\calF}(v)=
\inf \{\QA{s}(v) \colon s \in S\}$. However, by \thm{main}, there exists $u \in \Cs(I)$ such that $u \in S$ and $\QA{u} \le \QA{s}$ for all $s \in S$. Then we trivially obtain $\UQA{\calF}=\QA{u}$.

The proof of the second part is completely analogous. 
\end{proof}

Our last result generalizes Theorem~1 of \cite{Pas20b}.

\Cor{*}{
 Let $\calF \subset \Cts(I)$ and $x_0 \in \interior I$.\\
 (A) If $G \colon I \ni x \mapsto \sup_{f \in \calF} \frac{f''(x)}{f'(x)}$ is well-defined (real-valued) and integrable then $\UQA{\calF}=\QA{u}$, where
 \Eq{*}{
    u(x)=\int_{x_0}^x\exp\bigg(\int_{x_0}^t G(s)\:ds\bigg)dt.
 }
 (B) If $H \colon I \ni x \mapsto \inf_{f \in \calF} \frac{f''(x)}{f'(x)}$ is well-defined (real-valued) and integrable then $\LQA{\calF}=\QA{\ell}$, where
 \Eq{*}{
    \ell(x)=\int_{x_0}^x\exp\bigg(\int_{x_0}^t H(s)\:ds\bigg)dt.
 }
 }
\begin{proof}
We focus on part (A) since the second is analogous. Take $f_0 \in \calF$ arbitrarily. Since $\QA{-f_0}=\QA{f_0}$ we can assume without loss of generality that $f_0$ is increasing. 
Then $G-\frac{f''_0}{f'_0}\ge 0$. Therefore,
due to the integrability of $G$, the function $u\colon t \mapsto \int_{x_0}^t G(s)-\frac{f''_0}{f'_0}(s)ds$ is well-defined, increasing and continuous.  Hence, so is $\exp u$. Therefore, 
\Eq{*}{
\exp\bigg(\int_{x_0}^t G(s)\:ds\bigg)&=\exp\bigg(u(t)+\int_{x_0}^t \frac{f''_0}{f'_0}(s)\:ds\bigg)=\exp\bigg(u(t)+\ln f_0'(t)-\ln f_0'(x_0)\bigg)\\
&=\frac{f_0'(t)}{f_0'(x_0)}\exp(u(t))
}
is continuous, and therefore integrable. Thus $u$ is well-defined and differentiable. 

Now we show that $\UQA\calF =\QA{u}$. Due to \thm{FCs}, we know that $\UQA\calF =\QA{v}$ for some $v \in \Cs(I)$.
Moreover for all $f \in \calF$ we have $G \ge \frac{f''}{f'}$ and hence for all $x,y \in I$ with $x \le y$ we obtain 
\Eq{*}{
\log u'(y)-\log u'(x)=\int_x^y G(s)\:ds\ge  \int_x^y \frac{f''(s)}{f'(s)}\:ds=\log f'(y)-\log f'(x).
}
Thus $\log u'(y)-\log f'(y)\ge\log u'(x)-\log f'(x)$ and due to \rem{C1QA} we get $\QA{u}\ge \QA{f}$ for all $f \in \calF$. Hence $\QA{u}\ge \UQA\calF=\QA{v}$.

If $\QA{v} \not \ge \QA{u}$ then $\frac{v'}{u'}$ is not increasing, that is $\frac{v'}{u'}(q)<\frac{v'}{u'}(p)$ for some $p,q \in I$ with $p<q$. Due to the continuity, we may assume that none of these points is the endpoint of $I$. 
Thus $\log v'(q) - \log u'(q)<\log v'(p) - \log u'(p)$, which implies
\Eq{*}{
\log v'(q) - \log v'(p) < \log u'(q) - \log u'(p)=\int_p^q G(s)\:ds.
}
Therefore, there exists $\varepsilon>0$ such that
\Eq{*}{
\log v'(q) - \log v'(p) < \int_p^q G(s)-\varepsilon\:ds.
}
Now, for every $s \in [p,q]$ there exists $g_s \in \calF$ such that $(\frac{g_s''}{g_s'})(s)>G(s)-\varepsilon$. What is more, there exists an open neighbourhood $V_s \ni s$ such that $(\frac{g_s''}{g_s'})(x)>G(x)-\varepsilon$ for all $x \in V_s$. Then $(V_s)$ is a family of open intervals which covers $[p,q]$. Therefore, there exists a finite subfamily $(V_{s_1},\dots,V_{s_n})$ which also covers $[p,q]$ ($s_1<s_2<\dots<s_n$). We can assume that for each $i \in\{1,\dots,n-1\}$ the intersection
$V_{s_i} \cap V_{s_{i+1}}$ contains some element, say $m_i$. Additionally we set $m_0:=p$ and $m_n:=q$. Then we have $[m_{i-1},m_i]\subset V_{s_i}$ for all $i \in \{1,\dots,n\}$ and hence 
\Eq{*}{
\sum_{i=1}^n \log v'(m_i) - \log v'(m_{i-1}) 
&=\log v'(q) - \log v'(p)
< \int_p^q G(s)-\varepsilon\:ds
=\sum_{i=1}^n \int_{m_{i-1}}^{m_i} G(s)-\varepsilon\:ds\\
&<\sum_{i=1}^n \int_{m_{i-1}}^{m_i} \frac{g_{s_i}''(s)}{g_{s_i}'(s)}\:ds
=\sum_{i=1}^n \log g_{s_i}'(m_i)-\log g_{s_i}'(m_{i-1}).
}
Thus there exists $j \in\{1,\dots,n\}$ such that $\log v'(m_j) - \log v'(m_{j-1})<\log g_{s_j}'(m_j)-\log g_{s_j}'(m_{j-1})$. Equivalently $\log v'(m_j) -\log g_{s_j}'(m_j)< \log v'(m_{j-1})-\log g_{s_j}'(m_{j-1})$. Thus $\UQA{\calF}=\QA{v} \not \ge \QA{g_{s_j}}$, a contradiction.

Finally, we obtain $\QA{v} \ge \QA{u}$, and hence $\QA{u}=\QA{v}=\UQA\calF$. 
\end{proof}

\end{document}